\documentclass[12pt,leqno]{article}
\usepackage[shortlabels]{enumitem}
\usepackage{preamble}
\newcommand{\self}{\operatorname{self}}
\title{A Dirichlet problem in noncommutative potential theory\thanks{Research partially supported by NSF grant DMS--1464150.\newline
2010 Mathematics Subject Classification 31C05, 35J25, 47A56, 58J32}
}
\author{Kuang-Ru Wu}
\newcommand{\RN}[1]{%
  \textup{\uppercase\expandafter{\romannumeral#1}}%
}
\date{}

\usepackage{fullpage}

\usepackage{fancyhdr}

\theoremstyle{plain}
\newtheorem{thm}{Theorem}[section]
\newtheorem {lem}[thm]{Lemma}

\numberwithin{equation}{section}

\begin{document}

\parskip=6pt

\maketitle
\begin{abstract}
We prove the solvability of a Dirichlet problem for flat hermitian metrics on Hilbert bundles over compact Riemann surfaces with boundary. We also prove a factorization result for flat hermitian metrics on doubly connected domains.
\end{abstract}

\section{Introduction}
Let $(V,\gen{\blk,\blk})$ be a complex Hilbert space, $\End V$ the set of bounded linear operators on $V$, and $\End^+V$ the set of all positive invertible elements of $\End V$. Let $\overline{M}$ be a compact Riemann surface with boundary ($\partial M$ is automatically a real analytic manifold by the reflection principle). On the bundle $\overline{M}\times V \to \overline{M}$, a hermitian metric $h$ is a collection of hermitian inner products $h_z$ on $V$, for $z\in \overline{M}$, and it can be written  $h_z(v,w)=\gen{P(z)v,w}$ with $P: \overline{M} \to \End^+V$, $v$ and $w\in V$. Assume $P$ is $C^2$, the Chern connection of the metric is $P^{-1}\partial P$, and the curvature $R^P=\bar{\partial}(P^{-1}\partial P)=P^{-1}(P_{z\bar{z}}-P_{\bar{z}}P^{-1}P_z)d\bar{z}\wedge dz$ in a chart. In this paper, we will address the Dirichlet problem of extending a given metric on $\partial M \times V$ to a metric on $\overline{M}\times V$ that has zero curvature.  Our main result is the following:
\begin{thm}\label{thm:4}
Let $\overline{M}$ be a compact Riemann surface with boundary and $F\in C^{m}(\partial M, \End^+V)$, where $m=0,\infty, \text{or }\omega$. There exists a unique $P\in C^{m}(\overline{M}, \End^+V)\cap C^2(M,\End^+V)$ such that $R^P=0$ on $M$, and $P|_{\partial M}=F$. The same is true if we replace $C^m$ by $C^{k,\alpha}$ for $k$ a nonnegative integer and $0<\alpha<1$.
\end{thm}
We mention briefly previous work when $\dim V <\infty$. Masani and Wiener prove a factorization result in \cite{wiener1957prediction} which can be used to solve the Dirichlet problem over the unit disc, with regularity weaker than continuous. In \cite{MR660145}, Lempert proves this factorization to H\"older classes. More generally, in \cite{MR1165874}, Donaldson solves a Dirichlet problem for the Hermitian Yang--Mills equations over K\"{a}hler manifolds with boundary, and in \cite{MR1216432} Coifman and Semmes solve it over domains in $\mathbb{C}^n$ which are regular for the Laplacian. When the base is one dimensional, Donaldson's and Coifman-Semmes' results reduce to existence of flat hermitian metrics. (Coifman and Semmes also solve a Dirichlet problem for norms more general than those coming from hermitian metrics. See also a more recent related paper \cite{2016arXiv160706306B}.)

Devinatz \cite{devinatz1961factorization} and Douglas \cite{douglas1966factoring} generalize Wiener-Masani factorization to infinite dimensional separable $V$, with the base still the unit disc (see also \cite[Lecture XI]{MR0171178}). For a general $V$ and various regularity classes, the Dirichlet problem over the unit disc is solved by Lempert in \cite{MR3738363}. Lempert's proof is by the continuity method and proceed by a global factorization of flat metrics. However, such a factorization is not available when the base is multiply connected. 

Our proof is also by the continuity method. Closedness is proved by a maximum principle and a local holomorphic factorization of flat metrics. Openness turns out to be harder than usual, because to deal with the linear partial differential equation originating from the implicit function theorem, Fredholm theory is not available. However, the linear equation has various symmetries that we can exploit to obtain the requisite a priori estimates. For details, see section \ref{est}.

The structure of this paper is as follows. In section \ref{pre} and \ref{est}, we collect a few preliminary lemmas and provide a priori estimates for both the nonlinear equation $R^P=0$ and its linearization. In section \ref{sec:pf}, we prove Theorem \ref{thm:4}. In section \ref{sec:db}, we prove a global factorization for flat hermitian metrics on doubly connected domains, an analog of \cite[Theorem 3.1]{MR3738363}. We consider annuli for convenience. In what follows, $\End^{\times}V$ is the set of invertible elements in $\End V$, and $\End^{\self} V$ is the set of self-adjoint operators.
\begin{thm}\label{thm:1.2}
Let $M=\{z\in \mathbb{C}:r_1<|z|<r_2\}$, and $F\in C(\partial M,\End^+V)$. There exist $H\in \mathcal{O}(M,\End^{\times}V)$ and $a\in \End^{\self} V$ such that the function
\begin{equation*}
P(z)=
\begin{cases}
H^*(z) \exp(a\log|z|^2) H(z) &\text{for $z\in M$} \\
F(z)  &\text{for $z \in \partial M$}
\end{cases}
\end{equation*}
is in $C(\overline{M},\End^+ V)$. Moreover, if $F\in C^{k,\alpha}(\partial M,\End^+V)$ for $k$ a nonnegative integer and $0<\alpha<1$, then $H$ extends to a function in $C^{k,\alpha}(\overline{M},\End^{\times}V)$.
\end{thm}

Straightforward calculations show that $P$ in the above theorem has curvature 0. We conjecture similar factorizations to exist in $m$-connected domains. However, when $m$ is at least three, the fundamental group of $M$ is nonabelian, in addition to our already noncommutative operators, and we haven't been able to write down a meaningful factorization. Such factorizations might provide another proof of Theorem \ref{thm:4} without invoking a priori estimates, as in \cite{MR3738363}.

As for nontrivial Hilbert bundles, it is known that such bundles can be trivialized over open Riemann surfaces. 
It is likely that this is true also over Riemann surfaces with boundary, but we do not pursue such question in this paper. 

I would like to thank Chi Li, Hengrong Du, and Seongjun Choi for discussions and their critical comments, and Carlos Salinas for his suggestions on the presentation of this paper. I am grateful to L\'aszl\'o Lempert for introducing me to this problem, and for his constant encouragements, discussions, and inspirations.  

\section{Preliminary lemmas}\label{pre}
We will deal with spaces of maps with values in $\End V$, such as $C^{k,\alpha}(\overline{M},\End V)$, and we briefly indicate what they are. First, $\End V$ with the operator norm $||\cdot||_{\op}$ is a Banach space. Smoothness for maps with values in $\End V$ will always refer to this Banach space topology. $\mathcal{O}(M,\End V)$ denote the space of holomorphic maps, those that are complex differentiable in charts. Similarly, given a smooth manifold $\overline{N}$, possibly with boundary, if $k=0,1,2...$ and $0<\alpha<1$, $C^k(\overline{N},\End V)$ and $C^{k,\alpha}(\overline{N}, \End V)$ consist of maps that are $C^k$, respectively $C^{k,\alpha}$ in charts. These two can be given a Banach algebra structure, if $\overline{N}$ is compact and a finite open cover $\{U_i\}$ of $\overline{N}$ is fixed so that each $\overline{U_i}$ is contained in a chart. For $f\in C^{k,\alpha}(\overline{N},\End V)$, say, one just computes the corresponding H\"older norms in each $U_i$ using the local coordinates, and defines $||f||_{k,\alpha,N}$ as the sum of those H\"older norms. With a suitable scaling it can be arranged that $||\cdot||_{k,\alpha,N}$ is submultiplicative, namely, $C^{k,\alpha}(\overline{N}, \End V)$ is a Banach algebra. Similarly, $C^k(\overline{N},\End V)$ also carries a Banach algebra structure. For more details, see \cite[p.610]{MR3738363}.

We set $C^{\infty}=\cap_k C^k$, and also write $C$ for $C^0$. Finally, if $\overline{N}$ is a real analytic manifold, we denote by $C^{\omega}(\overline{N},\End V)$ the space of real analytic maps, those that can be expanded at each point of $\overline{N}$ in a power series in a chart.

In traditional potential theory, a real-valued harmonic function on a simply connected open set in $\mathbb{C}$ is the real part of a holomorphic function, unique up to a purely imaginary additive constant. There is a corresponding result in noncommutative potential theory.
\begin{lem}\label{l,3}
If $M$ is a simply connected Riemann surface and $P\in C^2(M, \End^+V)$ is flat, namely $R^P=0$, then $P=H^*H$ where $H\in \mathcal{O}(M,\End^{\times}V)$. If $P=K^*K$ is also such a factorization, then $H=UK$, where $U\in \End V$ is unitary.
\end{lem} 
If $\dim V$=1, we recover the traditional result by taking logarithms.
\begin{proof}
This lemma is actually true for $M$ a simply connected complex manifold, see \cite[Chapter \RN{5}.6]{demailly1997complex}. Although the bundle is of finite rank there, the proof carries over to infinite rank easily.
\end{proof}

\begin{lem}\label{lem:2}
Let $\overline{M}$ be a compact Riemann surface with boundary, $P_j\in C(\overline{M},\End^+V)\cap C^2(M, \End^+V)$, and $R^{P_j}=0$, $j\in \mathbb{N}$. If $P_j|_{\partial M}$ converges in $C(\partial M, \End^+V)$, then $P_j$ converges in $C(\overline{M}, \End^+V)$.
\end{lem}
\begin{proof}It is basically the same as the proof of \cite[Corollary 3.3]{MR3738363}.\end{proof}

\begin{lem}\label{lem:3}Let $D\subset \mathbb{C}$ be the unit disc, $H_j\in \mathcal{O}(D,\End^{\times}V)$, and $H_j(0)\in \End^+V$. If $H_j^*H_j$ converges to some $P\in C(D,\End^+V)$, then there exists $H\in \mathcal{O}(D,\End^{\times}V)$ such that $H_j$ converges, locally uniformly, to $H$ on $D$.\end{lem}

\begin{proof} See the proof of \cite[Theorem 3.1]{MR3738363}.\end{proof}

\section{A priori estimates}\label{est}                                                                                                                                                                                  Fix a smooth positive $(1,1)$-form $\omega$ on $\overline{M}$ and define a map $\Lambda$ sending $(1,1)$-forms to functions: $\Lambda(\phi)=-\phi/\omega$, for a $(1,1)$-form $\phi$. Locally, $\omega =\sqrt{-1}g dz\wedge d\bar{z}$, where $g$ is a positive smooth function, so if $\phi=vdz\wedge d\bar{z}$ locally, then $\Lambda (\phi)=\sqrt{-1}v/g$.  

Fix $0<\alpha<1$, assume $P\in C^{2,\alpha}(\overline{M},\End^+V)$ is flat, and $A=P^{-1}\partial P$. We associate the following differential operator with $P$:  
\begin{equation*}
L:C^{2,\alpha}(\overline{M},\End^{\self} V)\longrightarrow C^{\alpha}(\overline{M}, \End^{\self} V)$$ $$h\longmapsto \sqrt{-1}\Lambda(\bar{\partial}\partial h-A^*\wedge \partial h-\bar{\partial}h\wedge A+A^*\wedge h\wedge A).
\end{equation*}
On a chart, $Lh=(1/g) \mathscr{L}h$, where $\mathscr{L}h=h_{z\bar{z}}-P_{\bar{z}}P^{-1}h_z-h_{\bar{z}}P^{-1}P_z+P_{\bar{z}}P^{-1}hP^{-1}P_z$. The reason for studying $L$ is that it is the linearization of curvature, as we shall see in section \ref{sec:pf}. The main result in this section is
\begin{thm}\label{cor:11}
If $h\in C^{2,\alpha}(\overline{M},\End^{\self} V)$ and $h|_{\partial M}=0$, then $$\|h\|_{2,\alpha, M}\leq C\|Lh\|_{0,\alpha,M}$$ where $C=C(\|P\|_{2,\alpha},\|P^{-1}\|_{0,\alpha})$.
\end{thm}
We begin with a somewhat standard estimate.
\begin{lem}\label{lem:9}
If $h\in C^{2,\alpha}(\overline{M},\End^{\self} V)$ and $h|_{\partial M}=0$, then$$\|h\|_{2,\alpha, M}\leq C(\|h\|_{0,M}+\|Lh\|_{0,\alpha,M})$$ where $C=C(\|P\|_{2,\alpha},\|P^{-1}\|_{0,\alpha})$.
\end{lem}
The prominent feature of $L$ is the following. On a simply connected open set, we have $H^*PH=1$ with holomorphic $H$ by Lemma \ref{l,3}, and it turns out that $$\frac{1}{2}\Delta(H^*hH)= -H^*(Lh)H.$$
Here $\Delta$ is the Laplace operator with respect to $\omega$, and we use the fact that $\Delta$ when acting on functions is the same as $2\sqrt{-1}\Lambda \bar{\partial}\partial$. Therefore, modulo a gauge transformation $H$, $L$ is the Laplace operator, locally. In a chart, the above equality becomes $(H^*hH)_{z\bar{z}}=H^*(\mathscr{L}h)H$. We will exploit this to reduce Lemma \ref{lem:9} to the corresponding estimates for scalar-valued elliptic partial differential equations. 

If $L$ had nonpositive zero order term, general theory would imply $\|h\|_{0, M}\leq C\|Lh\|_{0, M}$, which together with Lemma \ref{lem:9} would give Theorem \ref{cor:11}. Nonetheless, the zero order term of $L$ has the opposite sign. To get around this problem we first prove a maximum principle, Lemma \ref{lem:12}, and observe that for $u\in C^2(\overline{M},\mathbb{C})$, $$L(u\cdot P)=\sqrt{-1}\Lambda (-\bar{\partial}\partial u\cdot P+uP\cdot R^P)=(-\frac{1}{2}\Delta u)P,$$ as $R^P=0$. A suitable choice of $u$ will put us in the position of using Lemma \ref{lem:12}, and Theorem \ref{cor:11} will follow quickly. 
\begin{proof}[Proof of Lemma \ref{lem:9}]
Consider two finite open covers $\{U_i\},\{V_i\}$ of $\overline{M}$, such that $\overline{U_i},\overline{V_i}$ are in a chart $\phi_i$ for each $i$, and 
\begin{equation*}
\begin{aligned}
\text{for interior chart, }&\begin{cases}
\phi_i(U_i)=B(0,1)\\
\phi_i(V_i)=B(0,2).  
\end{cases}\\
\text{for boundary chart, }
&\begin{cases}
\phi_i(U_i)=B(0,1)\cap \overline{H}\\
\phi_i(V_i)=B(0,2)\cap \overline{H}
\end{cases}
\text{where $H\subset \mathbb{C}$ is the upper-half plane.}
\end{aligned}
\end{equation*}
We use $\{U_i\}$ to define the norm on $C^{2,\alpha}(\overline{M},\End^{\self} V)$ and $\{V_i\}$ on $C^{\alpha}(\overline{M},\End^{\self} V)$. 

Since our arguments will be local, we can assume $U_i,V_i$ are already in $\mathbb{C}$ and $\phi_i$ is the identity. We first consider a boundary chart $\phi_i$. As mentioned above, $(H^*hH)_{z\bar{z}}=H^*(\mathscr{L}h)H$, where $H$ is a holomorphic function in the interior of this chart with $H^*PH=1$. As $P$ is $C^{2,\alpha}$ up to boundary of $\overline{M}$, so is $H$, according to \cite[Theorem 3.7]{MR3738363}. Consider a bounded linear functional $l\in (\End V)^*$ of norm one, and apply $l$ to the equation obtaining $[l(H^*hH)]_{z\bar{z}}=l(H^*(\mathscr{L}h)H)$, a scalar-valued equation. Denote $\phi_i(U_i)=B'$ and $\phi_i(V_i)=B''$. By  \cite[Lemma 6.5 or Corollary 6.7]{MR1814364}
\begin{equation}\label{eq:1}
\begin{aligned}\|l(H^*hH)\|_{2,\alpha, B'}\leq C(\|l(H^*hH)\|_{0,B''}+\|l(H^*(\mathscr{L}h)H)\|_{0,\alpha,B''})
\end{aligned}
\end{equation}
where $C$ is a uniform constant.
 We can get rid of $l$ and $H$ to have $$\|h\|_{2,\alpha, B'}\leq C(\|h\|_{0,M}+\|Lh\|_{0,\alpha,M}).$$

Indeed, at each point in $B''$, $$|l(H^*hH)|\leq \|H^*hH\|_{\op}\leq \|H\|_{\op}^2\|h\|_{\op}\leq C\|h\|_{\op}.$$ The last inequality follows from $P^{-1}=HH^*$. Similarly,
\begin{align*}
\|l(H^*(\mathscr{L}h)H)\|_{0,\alpha,B''} &\leq \|H^*(\mathscr{L}h)H\|_{0,\alpha,B''}\\
&\leq \|H\|^2_{0,\alpha,B''}\cdot \|\mathscr{L}h\|_{0,\alpha,B''}\\
&\leq C\|H\|^2_{0,\alpha,B''}\cdot \|Lh\|_{0,\alpha,M}\\
&\leq C(\|H\|^2_{0}+\|H_z\|^2_{0})\cdot \|Lh\|_{0,\alpha,M}\\
&\leq C\|Lh\|_{0,\alpha,M}
\end{align*}
The third inequality is by the definition of the $C^{\alpha}$ norm on $M$. The last inequality follows from $H_z=-P^{-1}P_zH$. Therefore, the right hand side of (\ref{eq:1}) is dominated by $C(\|h\|_{0,M}+\|Lh\|_{0,\alpha,M})$. Namely, 
\begin{align*}
C(\|h\|_{0,M}+\|Lh\|_{0,\alpha,M})\geq &\|l(H^*hH)\|_{2,\alpha, B'}
\\
=&\|l(H^*hH)\|_{0}+\|Dl(H^*hH)\|_{0}+\|D^2l(H^*hH)\|_{0,\alpha},
\end{align*}
where $D$ stands for first order and $D^2$ for second order derivatives. Hence, for $x\in B'$, $$|Dl(H^*hH)(x)|\leq C(\|h\|_{0,M}+\|Lh\|_{0,\alpha,M})$$ and by the Hahn--Banach Theorem $$\|D(H^*hH)(x)\|_{\op}\leq C(\|h\|_{0,M}+\|Lh\|_{0,\alpha,M}).$$ As a consequence, 
\begin{align*} C(\|h\|_{0,M}+\|Lh\|_{0,\alpha,M}) &\geq\|DH^*hH+H^*DhH+H^*hDH \|_{0,B'}\\
&\geq \|H^*DhH\|_{0,B'}-\|H^*hDH \|_{0,B'}-\|DH^*hH\|_{0,B'}\\
&\geq \|H^*DhH\|_{0,B'}-C\|h\|_{0,B'}.
\end{align*}
So $$C(\|h\|_{0,M}+\|Lh\|_{0,\alpha,M})\geq \|H^*DhH\|_{0,B'}.$$ Since 
\begin{align*}
\|Dh\|_{0,B'}&\leq \|H^*DhH\|_{0,B'}\|{H}^{-1}\|_{0,B'}^2=\|H^*DhH\|_{0,B'}\|P\|_{0,B'}\leq C\|H^*DhH\|_{0,B'},
\end{align*}
we have $$\|Dh\|_{0, B'}\leq C(\|h\|_{0,M}+\|Lh\|_{0,\alpha,M}).$$ 
We can estimate the second derivatives and their H\"older norms similarly, and obtain
\begin{equation}\label{eq:2}
\|h\|_{2,\alpha, B'}\leq C(\|h\|_{0,M}+\|Lh\|_{0,\alpha,M}).
\end{equation}
 
We next consider an interior chart $\phi_i$. As before $[l(H^*hH)]_{z\bar{z}}=l(H^*(\mathscr{L}h)H)$. We let $\phi_i(U_i)=B'$ and $\phi_i(V_i)=B''$. By \cite[Corollary 6.3]{MR1814364}, 
\begin{equation*}
\begin{aligned}
\|Dl(H^*hH)\|_{0,B'}+\|D^2l(H^*hH)\|_{0,B'}+[D^2l(H^*hH)]_{\alpha,B'}\\
\leq C\big[\|l(H^*hH)\|_{0,B''}+\|l(H^*(\mathscr{L}h)H)\|_{0,\alpha, B''}\big]. 
\end{aligned}
\end{equation*}
Using the same method as in boundary charts, we can get rid of $l$ and $H$ to obtain the same estimate (\ref{eq:2}). Hence the lemma follows.
\end{proof}

We next prove a maximum principle, which in turn gives rise to $C^0$ estimates. Recall that $\gen{\blk,\blk}$ is the inner product of $V$, and denote $\|v\|^2_{P(z)}= \langle P(z)v,v\rangle$.

\begin{lem}\label{lem:12}
Suppose $h\in C^2({M}, \End^{\self} V)$. Define $$S_{P,h}(z)=\sup_{\|v\|_{P(z)}=1}\langle h(z)v,v\rangle.$$  If $Lh\geq 0$, then $S_{P,h}(z)$ is subharmonic. As a result, if additionally $h$ is continuous on $\overline{M}$, then $$\sup_{\overline{M}}S_{P,h}=\sup_{\partial M}S_{P,h}.$$
\end{lem} 
\begin{proof}
First, 
\begin{align*} S_{P,h}(z)&=\sup_{\langle P(z)v,v\rangle=1}\langle h(z)v,v\rangle=\sup_{\|{P(z)}^{1/2}v\|=1}\langle{P(z)}^{-1/2}h(z){P(z)}^{-1/2}{P(z)}^{1/2}v,{P(z)}^{1/2}v\rangle\\
&=\sup_{\|u\|=1}\langle{P(z)}^{-1/2}h(z){P(z)}^{-1/2}u,u\rangle
\end{align*}
is continuous, as the sup of a family of equicontinuous functions. Locally, we have $H^*PH=1$ and $(H^*hH)_{z\bar{z}}=H^*(\mathscr{L}h)H$; furthermore, $0\leq Lh=(1/g)\cdot \mathscr{L}h$ means $\mathscr{L}h\geq 0$. Since $$0\leq \langle(\mathscr{L}h)Hv,Hv\rangle=\langle(H^*hH)_{z\bar{z}}v,v\rangle,$$ $\langle(H^*hH)v,v\rangle$ is subharmonic for any $v\in V$. Thus, \begin{align*}
S_{P,h}(z)=\sup_{\langle P(z)v,v\rangle=1}\langle h(z)v,v\rangle =\sup_{\langle H^{-1}v,H^{-1}v\rangle=1}\langle h(z)v,v\rangle =\sup_{\langle u,u\rangle=1}\langle H^*hH(z)u,u\rangle
\end{align*}
is the sup of a family of subharmonic functions. As we already know $S_{P,h}(z)$ is continuous, it is subharmonic. 
\end{proof}

\begin{thm}\label{thm:10}
If $h\in C^{2,\alpha}(\overline{M},\End^{\self} V)$ and $h|_{\partial M}=0$, then $$\|h\|_{0, M}\leq C\|Lh\|_{0, M}$$ where $C=C(\|P\|_0, \|P^{-1}\|_0)$.
\end{thm}
\begin{proof}
Recall if $u\in C^2(\overline{M},\mathbb{C})$, then $$L(u\cdot P)=(-\frac{1}{2}\Delta u)P.$$
Let $\Phi$ be the function vanishing on $\partial M$ such that $\Delta \Phi =2$, and let $G=(\Phi-\inf \Phi)\|P^{-1}\|_0P$. Then $G\geq 0$ with $L(G)=-\|P^{-1}\|_0P\leq -1$. Besides, $G\leq C$, where $C$ depends on $\|P\|_0$ and $\|P^{-1}\|_0$. With $F=G\cdot \|Lh\|_{0}$, we have $h\leq F$ on $\partial M$. Moreover,
\begin{align*}
L (h-F)&=Lh-\|Lh\|_{0}\cdot LG\geq Lh+\|Lh\|_{0}\geq 0.
\end{align*}
By Lemma \ref{lem:12}, $h-F\leq 0$ on $M$. Therefore, $$h\leq G\cdot \|Lh\|_{0}\leq C\|Lh\|_{0}.$$ Replacing $h$ by $-h$, the theorem follows.
\end{proof}
Theorem \ref{cor:11} is a consequence of Lemma \ref{lem:9} and Theorem \ref{thm:10}.
\section{Proof of Theorem \ref{thm:4}}\label{sec:pf}
We start with a regularity result.
\begin{lem}\label{lem R}
Let $P\in C(\overline{M}, \End^+V)\cap C^2(M,\End^+V)$ be flat. If $P|_{\partial M}$ is $C^{k,\alpha}$, $C^{\infty}$, or $C^{\omega}$, then $P$ has the corresponding regularity on $\overline{M}$.
\end{lem}
\begin{proof}
By Lemma \ref{l,3}, $P=H^*H$ with a holomorphic map $H$ locally, so $P$ is always $C^{\omega}$ in $M$ regardless of its boundary values. Denote $P|_{\partial M}$ by $F$. Suppose $F\in C^{k,\alpha}$, we have on a boundary chart $P=H^*H$, and $H$ is $C^{k,\alpha}$ up to $\partial M$ by \cite[Theorem 3.7]{MR3738363}; therefore, $P$ is $C^{k,\alpha}$ up to $\partial M$. Next suppose $F$ is $C^{\infty}$, then by the $C^{k,\alpha}$ result, $P$ is $C^k$ up to $\partial M$ for any positive integer $k$, hence $C^{\infty}$.

Finally, suppose $F\in C^{\omega}$. On a boundary chart, that we identify with the upper-half disc in $\mathbb{C}$, $P=H^*H$ with $H$ continuous up to the real axis by \cite[Theorem 3.7]{MR3738363}. Since $F\in C^{\omega}$, it has a holomorphic extension in a neighborhood of the real axis in the disc, so the map ${H^*}^{-1}(\bar{z})\cdot F(z)$ provides $H$ a holomorphic extension across the real axis, it follows that $P$ is real analytic across the real axis.
\end{proof}
\begin{proof}[Proof of Theorem \ref{thm:4}]
The uniqueness follows from the maximum principle (see \cite[Lemma 3.2]{MR3738363} or \cite{MR3314125}). We consider first the case $F\in C^{\omega}$ and prove the existence by the continuity method. Fix $0<\alpha<1$, let $\phi_{t}=tF+(1-t)\textrm{Id}$, and 
$$T=\left\{ t\in[0,1] \:\middle|\: 
\begin{gathered}
\text{ If } 0\leq s\leq t,
\text{ then } \phi_s= P_s|_{\partial M}, \\
\text{ for some } P_s\in C^{2,\alpha}(\overline{M},\End^+V), \text{ and } R^{P_s}=0 
\end{gathered}
\right\}.$$ 
We will say those $\phi_s$ ``have an extension.'' The goal is to show $T=[0,1]$. If so, $\phi_{1}=F$ has a $C^{2,\alpha}$ extension, and we can improve the regularity from $C^{2,\alpha}$ to $C^{\omega}$ by Lemma \ref{lem R}. Because 0 is in $T$, $T$ is nonempty. First we prove $T$ is closed. 

Suppose $T \ni t_j\to t_0$. For $s<t_0$, we can find $t_j>s$, therefore $\phi_s$ has an extension. We have to show $\phi_{t_0}$ extends. For brevity, we write $P_j$ instead of $P_{t_j}$. Since $P_j|_{\partial M }=\phi_{t_j} \to \phi_{t_0}$, $P_j$ converges by Lemma \ref{lem:2}, say to $P_{\infty}\in C(\overline{M}, \End^+V)$, and $P_{\infty}|_{\partial M}=\phi_{t_0}$. For any interior point of $\overline M$, choose a chart with image the unit disc $D$ in $\mathbb{C}$. Thus, $P_j=H^*_jH_j$, where $H_j\in \mathcal{O}(D, \End^{\times}V)$ by Lemma \ref{l,3}, and after multiplying with the unitary operator $(H^*_j(0)H_j(0))^{1/2}H^{-1}_j(0)$ we can assume $H_j(0)\in \End^+V$. By Lemma \ref{lem:3}, there exists $H$ holomorphic on $D$ such that $H_j\to H$ locally uniformly. Hence, $P_{\infty}=\lim H_j^*H_j = H^*H$ on $D$ which implies $P_{\infty}\in C^{\infty}(M,\End^+V)$ and $R^{P_{\infty}}=0$. By Lemma \ref{lem R}, $P_{\infty}$ is $C^{\omega}$, especially $C^{2,\alpha}$ on $\overline{M}$. Hence, $t_0$ is in $T$ and $T$ is closed.

Now we prove that $T$ is open. If $t_0\in T$ then $\phi_t$ has an extension $P_t$, for $0\leq t\leq t_0$. Consider the smooth map 
\begin{align*}
\Psi:C^{2,\alpha}(\overline{M},\End^+V)&\to C^{\alpha}(\overline{M}, \End^{\self} V)\times C^{2,\alpha}(\partial M,\End^+V)\\ 
h&\mapsto (\sqrt{-1}\Lambda (h\bar{\partial}(h^{-1}\partial h)),h|_{\partial M}).
\end{align*}
Then $\Psi(P_{t_0})=(0,\phi_{t_0} )$. We denote $P^{-1}_{t}\partial P_{t}=A_{t}$, so the linearization of $\Psi$ at $P_{t_0}$ is 
\begin{align*}
C^{2,\alpha}(\overline{M},\End^{\self} V)&\to C^{\alpha}(\overline{M}, \End^{\self} V)\times C^{2,\alpha}(\partial M,\End^{\self} V)\\
h&\mapsto (\sqrt{-1}\Lambda(\bar{\partial}\partial h-A^*_{t_0}\wedge \partial h-\bar{\partial}h\wedge A_{t_0}+A^*_{t_0}\wedge h\wedge A_{t_0}),h|_{\partial M}).
\end{align*}
It is here the operator in section \ref{est} turns up. We will show that the linearization is an isomorphism. Then $\Psi$ is a diffeomorphism in a neighborhood of $P_{t_0}$ by the implicit function theorem, and that implies $T$ is open. 

To show that the linearization is an isomorphism, it suffices to prove it is bijective because of the Open Mapping Theorem. That is, given $$(f_1,f_2)\in C^{\alpha}(\overline{M}, \End^{\self} V)\times C^{2,\alpha}(\partial M,\End^{\self} V),$$ the equation 
\begin{equation}\label{eq4.1}
\begin{cases}
\sqrt{-1}\Lambda(\bar{\partial}\partial h-A^*_{t_0}\partial h-\bar{\partial}h A_{t_0}+A^*_{t_0}hA_{t_0})=f_1\\ h|_{\partial M}=f_2

\end{cases}
\end{equation}
has a unique solution. That there is at most one solution easily follows from the maximum principle, Lemma \ref{lem:12} or Theorem \ref{cor:11}. If $\dim V <\infty$, existence follows from uniqueness by Fredholm alternative. However, if $\dim V=\infty$, Fredholm alternative is not available, because the embedding $C^{2,\alpha}(\overline{M},\End V)\rightarrow C^{\alpha}(\overline{M},\End V)$ is no longer compact. The way we solve (\ref{eq4.1}) is again the continuity method, based on the next lemma:
\begin{lem}\label{l,14}
Let $B,V$ be two Banach spaces, and $\{L_t\}_{0\leq t\leq 1}$ a family of bounded linear operators from $B$ to $V$. Suppose $t\mapsto L_t$ is continuous in operator norm; moreover, there exists a constant $C$ such that $\|x\|\leq C\|L_tx\|$ for any $x\in B$ and any $t$. Then $L_1$ is onto if and only if $L_0$ is onto.\end{lem}

This is a variant of \cite[Theorem 5.2]{MR1814364}. The proof is almost the same, so we skip it. Unsurprisingly, we are going to deform our equation to the Laplace equation. The naive way of deforming is by convex combination, but this breaks the symmetry of our equation (after all we want to use the a priori estimates from Theorem \ref{cor:11}). It is here the solution set $T$ plays its role; it tells us how to deform. 

First in equation (\ref{eq4.1}), $f_2$ can be extended to $C^{2,\alpha}(\overline{M},\End^{\self} V)$. If we subtract $f_2$ from $h$, we only need to consider the case of zero boundary value. In other words, we have to show that 
\begin{equation}\label{eq4.2}
L_t :\{h\in C^{2,\alpha}(\overline{M},\End^{\self} V):h|_{\partial M}=0\}\to C^{\alpha}(\overline{M}, \End^{\self} V)$$ $$h\mapsto \sqrt{-1}\Lambda(\bar{\partial}\partial h-A^*_{t}\wedge \partial h-\bar{\partial}h\wedge A_{t}+A^*_{t}\wedge h\wedge A_{t}) 
\end{equation}
is surjective when $t=t_0$. Note that $L_0$ is the Laplace operator, for $P_0=1$. We start with the following lemma, which is stronger than what we need.
\begin{lem}\label{15}
Let $k$ be a nonnegative integer. If $t,s\in [0,t_0]$ and $t\to s$, then $\|P_t-P_s\|_{C^k}\to 0$, and $\|{P_t}^{-1}-{P_s}^{-1}\|_{C^k}\to 0$.
\end{lem}
\begin{proof} 
By Lemma \ref{lem R}, $P_t\in C^k(\overline{M}, \End^+V)$.
Since $P_t|_{\partial M}=\phi_t\to \phi_s$, $P_t$ converges to $P_s$ in $C(\overline{M},\End^+V)$ by Lemma \ref{lem:2}. For the derivatives, we do estimates on charts and consider $\partial_z$ only, as $\partial_{\bar{z}}$ can be done in the same way. On an interior chart, $P_t=H^*_tH_t, P_s=H^*H$ where $H_t, H$ are holomorphic. As in the proof of closedness, $H_t\to H$ locally uniformly, and so do all their derivatives. Therefore, $$(P_t)_{z}=H^*_t(H_t)_{z} \to H^*H_z=(P_s)_{z}$$ locally uniformly. On a boundary chart, that again we identify with the upper-half disc in $\mathbb{C}$, we similarly have $(P_t)_{z}\to (P_s)_{z}$ locally uniformly but only away from the boundary. The convergence near the boundary can be resolved as follows. $P_t=H^*_tH_t$ with $H_t$ continuous up to boundary of $\overline{M}$ (in the current situation, this means the real axis in the disc) by \cite[Theorem 3.7]{MR3738363}. Similarly, $P_s=H^*H$ with $H$ contiunous up to boundary of $\overline{M}$. As in the proof of Lemma \ref{lem R}, since $\phi_t$ is $C^{\omega}$, it has a holomorphic extension in a neighborhood of the real axis in the disc, the map $${H^*_t}^{-1}(\bar{z})\cdot \phi_t(z)$$ provides an analytic continuation of $H_t$ across the real axis that we continue denoting $H_t$. For a compact set in the disc, consider a contour around it. By Cauchy's Integral Formula and the fact $\|H_t\|$ has a uniform upper bound, the Bounded Convergence Theorem implies that $H_t$ converges to $H$ uniformly on this compact set, and the same holds for derivatives of all orders. Hence, $P_t\to P_s$ in $C^k$ for any nonnegative integer $k$, locally uniformly in this boundary chart. Therefore, we conclude the $C^k$ convergence on $\overline{M}$. Since $C^k(\overline{M},\End V)$ is a Banach algebra, ${P_t}^{-1}\to {P_s}^{-1}$ in $C^k$.
\end{proof}
This lemma implies that $\|L_t-L_s\|\to 0$ as $t\to s$, where the norm on $L_t$ is the operator norm from (4.2). From Theorem \ref{cor:11} and the continuity Lemma \ref{15}, we get the desired estimates: if $h\in C^{2,\alpha}(\overline{M},\End^{\self} V)$ and $h|_{\partial M}=0$, then $$\|h\|_{2,\alpha, M}\leq C\|L_th\|_{0,\alpha, M}$$ where $C$ is independent of $t$. Therefore, by Lemma \ref{l,14} and the fact $L_0=\Delta /2$ is onto, $L_{t_0}$ is also onto, which implies the equation (\ref{eq4.1}) is uniquely solvable, so $T$ is open and therefore $T=[0,1]$. This completes the proof of Theorem \ref{thm:4} for $C^{\omega}$ case.

If the boundary data $F$ is only $C^0$, it can be approximated by a sequence $F_j\in C^{\omega}(\partial M, \End^+ V)$ in sup norm, for the following reason: $\partial M$ as a real analytic manifold can be real analytically embedded in some $\mathbb{R}^N$ by an embedding theorem of Grauert and Morrey \cite{MR0098847} \cite{10.2307/1970048}; $F$ has a continuous extension to $\mathbb{R}^N$, which can be approximated by polynomials $P_j$; after composing $P_j$ with the embedding, we have the desired $F_j$. Each $F_j$ has a real analytic flat extension $P_j$ according to the $C^{\omega}$ case. By Lemma \ref{lem:2}, $P_j$ converges in $C(\overline{M}, \End^+V)$, say to $P$. As in the proof of closedness, $P$ is $C^2$ in the interior and has curvature $0$.

If $F$ is $C^{k,\alpha}$ or $C^{\infty}$, the $P$ constructed in the previous paragraph is $C^{k,\alpha}$, respectively $C^{\infty}$ on $\overline{M}$, by Lemma \ref{lem R}. 
\end{proof}

\section{Factorization in doubly connected domains}\label{sec:db}

We prove Theorem \ref{thm:1.2} in this section. 
\begin{proof}
There exists a flat $P\in C(\overline{M}, \End^+V)$ with $P|_{\partial M}=F$ by Theorem \ref{thm:4}. The exponential map $e^{2\pi i z}$ is a universal covering map from the strip $$\{z\in \mathbb{C}:-\log r_2/2\pi <\Im(z)< -\log r_1/2\pi \}$$  to $M$. The composition $P(e^{2\pi i z})$ is flat on the strip, so by Lemma \ref{l,3}, $P(e^{2\pi iz})=H^*(z)H(z)$ where $H$ is holomorphic in the strip. Since $P(e^{2\pi iz})$ has period 1,
$$H^*(z+1)H(z+1)=H^*(z)H(z).$$This implies $${H^*}^{-1}(z+1)H^*(z)=H(z+1)H^{-1}(z).$$
But in the last equality, one side is holomorphic, the other is antiholomorphic, so both must be a constant, say $U$. Moreover, $$(U^*)^{-1}={[H(z+1)H^{-1}(z)]^*}^{-1}={H^*}^{-1}(z+1)H^*(z)=U,$$ so $U$ is unitary, and we have $UH(z)=H(z+1)$.

By Borel functional calculus (for example, see \cite[Chapter 12]{MR1157815}), $U=e^{iA}$ where $A\in \End^{\self}V$. Define $$K(z)=\exp\left(-iAz\right)\cdot H(z).$$ We have 
\begin{equation*}
\begin{aligned}
K(z+1)&=\exp\left(-iA(z+1)\right)\cdot H(z+1)=\exp\left(-iAz-iA\right)\cdot UH(z)\\
&=\exp\left(-iAz\right)\cdot H(z)=K(z).
\end{aligned}
\end{equation*}
That is, $K(z)$ is periodic.

As a result, $$P(w)=H^*H(\log w/2\pi i)=K^*(\log w/2\pi i)\cdot\exp\left[A(\log |w|^2/2\pi)\right]\cdot K(\log w/2\pi i).$$ Here $K(\log w/2\pi i)$ is single-valued, because $K$ is periodic. Since $A/2\pi$ is self adjoint, we have the desired factorization. (If $F$ is $C^{k,\alpha}$, then so are $P$ and $H$, therefore also $K$.)

\end{proof}

\bibliographystyle{amsalpha}
\bibliography{dom.bib}

\end{document}